\documentclass[11pt]{article}

\usepackage[margin=1in]{geometry}
\usepackage{amsmath, amssymb, amsthm}
\usepackage{algorithm, algpseudocode}
\makeatletter
\algnewcommand\LState[1]{%
  \State \parbox[t]{\dimexpr\linewidth-\ALG@thistlm\relax}%
         {\raggedright #1\strut}}
\makeatother
\usepackage{booktabs}
\usepackage{hyperref}
\usepackage{xcolor}
\usepackage{listings}

\newtheorem{theorem}{Theorem}
\newtheorem{proposition}[theorem]{Proposition}
\newtheorem{corollary}[theorem]{Corollary}

\title{Two-Thread Coverage MCTS for SAT and XSAT}
\author{Marcel Crasmaru\\\small\texttt{crasmarum@gmail.com}}
\date{\today}

\begin{document}
\maketitle

\begin{abstract}
We introduce a Monte-Carlo Tree Search solver for SAT and XSAT that
pairs WalkSAT-style rollouts with a two-thread symmetry-breaking
initialisation: one thread starts from all-true, the other from
all-false, bounding each thread's initial Hamming distance to a
satisfying assignment by $\lfloor n/2 \rfloor$. Empirically the
solver closes 100/100 SATLIB graph-colouring encodings
(\texttt{flat200-479}, \texttt{sw100}) in tens of milliseconds each,
20/20 planted 3-XOR-SAT at $n{=}200$ (median 10~s), 20/20 at
$n{=}300$ (median 90~s), and one SAT Competition 2025 instance
(\texttt{Break\_triple\_04\_06.xml.cnf}) in 50~ms via the polarity
split alone. On a full 16-round DES key-recovery encoding ($n{=}1976$,
$m{=}30\,072$) it drives the negative-clause count from $\sim 200$
down to 24 (99.9\% clauses satisfied) over 7 hours before hitting the
S-box plateau.
\end{abstract}

\section{Introduction}
\label{sec:intro}

Local search solvers based on WalkSAT~\cite{selman:walksat} solve
uniform-random 3-SAT far beyond the reach of complete methods but
scale poorly on structured or XOR-heavy formulas. MCTS as popularised
by AlphaGo~\cite{silver:alphago} orchestrates exploration versus
exploitation through UCT selection~\cite{kocsis:uct}. We combine the
two: MCTS drives diversification via a persistent tree with
stall-triggered restarts, while WalkSAT-style rollouts do the
intensification work at each leaf.

Two design choices distinguish the solver from prior local-search /
MCTS hybrids. First, the search is run in parallel from two
complementary starting assignments; we prove in
Section~\ref{sec:coverage} that this guarantees a satisfying
assignment reachable within $\lfloor n/2 \rfloor$ flips of one of them.
Second, XOR clauses are treated natively via a polymorphic negativity
predicate, and unit propagation is applied at three points — an
initial pass over the loaded formula (preprocess), the first evaluation
of each newly-expanded tree leaf's move (via a forced-flip cascade),
and every step inside every rollout. Note that the UCT descent phase
itself does \emph{not} re-run propagation: existing tree nodes replay
their memoised flip chains rather than recomputing them. State is
mutated \emph{incrementally} during descent — each node applies only
its own memoised flips, not the whole chain from root, so descending
a path totalling $k$ variable flips costs $\Theta(k)$ rather than
$\Theta(k^2)$. Correspondingly,
nodes store their per-move flip list ($O(\text{flips at that node})$),
not a snapshot of the whole assignment — a time-vs-memory trade in
favour of memory since $k$ can reach $n/2$.
Section~\ref{sec:algorithm} gives the algorithm.
Section~\ref{sec:results} reports empirical results across five
problem families plus the SAT 2025 and DES anecdotes.

\section{Preliminaries}
\label{sec:preliminaries}

\subsection{Formulas and clauses}

Let $V = \{x_1,\dots,x_n\}$ be a set of Boolean variables. A
\emph{literal} $\ell$ is either $x_i$ (positive) or $\neg x_i$
(negative). Under an assignment $\sigma \in \{0,1\}^n$, the value of
$\ell$ is $\sigma_i$ if $\ell = x_i$ and $1-\sigma_i$ if
$\ell = \neg x_i$.

We consider two kinds of clauses over literals $\ell_1,\dots,\ell_k$:
\begin{align}
\text{OR clause:}  \quad & C^\vee_{\ell_1,\dots,\ell_k}(\sigma)
    \;=\; \ell_1 \vee \ell_2 \vee \cdots \vee \ell_k, \\
\text{XOR clause:} \quad & C^\oplus_{\ell_1,\dots,\ell_k}(\sigma)
    \;=\; \ell_1 \oplus \ell_2 \oplus \cdots \oplus \ell_k.
\end{align}
A formula $F$ is a conjunction of clauses; $\sigma$ \emph{satisfies}
$F$ iff every clause is $1$ under $\sigma$.

We call a clause \emph{negative} under $\sigma$ when it evaluates to
$0$. Writing $|\ell|_\sigma \in \{0,1\}$ for the value of $\ell$ under
$\sigma$:
\begin{align}
C^\vee \text{ is negative under } \sigma \;\iff\;
    & |\ell_j|_\sigma = 0 \text{ for every } j,
\label{eq:orNeg}\\
C^\oplus \text{ is negative under } \sigma \;\iff\;
    & \bigoplus_{j=1}^{k} |\ell_j|_\sigma = 0.
\label{eq:xorNeg}
\end{align}

Let $\mathrm{neg}(F,\sigma)$ denote the number of clauses of $F$ that
are negative under $\sigma$.

\begin{proposition}[Satisfiability via negative-clause count]
\label{prop:satIffZero}
$F$ is satisfied by $\sigma$ iff $\mathrm{neg}(F,\sigma) = 0$.
\end{proposition}
\begin{proof}
Direct from the definitions. $F$ is a conjunction of clauses, so
$F(\sigma) = 1$ iff every clause evaluates to $1$ iff no clause is
negative.
\end{proof}

Proposition~\ref{prop:satIffZero} anchors the whole algorithm: the
search's objective is to drive $\mathrm{neg}(F,\sigma)$ to zero by
picking a suitable $\sigma$.

\subsection{Variable flipping}

The search moves in the assignment space $\{0,1\}^n$ by \emph{flipping
one variable at a time}: from $\sigma$ to $\sigma \oplus e_i$, where
$e_i$ has a $1$ in position $i$. A \emph{flipping path} of length $k$
is a sequence
$\sigma^{(0)} \to \sigma^{(1)} \to \cdots \to \sigma^{(k)}$ where
consecutive assignments differ in exactly one coordinate. The Hamming
distance $d_H(\sigma^{(0)}, \sigma^{(k)})$ equals the number of
variables flipped an odd number of times along the path.

Implementation-wise, each variable owns a list of the clauses it
appears in, and flipping $x_i$ walks that list and toggles the literal
sign in each clause. A per-clause \emph{negatives} counter maintains
$|\{j : |\ell_j|_\sigma = 0\}|$ incrementally, so
\eqref{eq:orNeg}--\eqref{eq:xorNeg} become $O(1)$ tests and
$\mathrm{neg}(F,\sigma)$ is maintained via a bitset of negative-clause
indices. The cost of a flip is $O(k \cdot d)$ where $d$ is the average
degree of a variable in the clause list ($d \approx 3$ for random 3-SAT).

\paragraph{Worked example: a three-flip walk to SAT.}
Take $n = 3$ and the OR-formula
\begin{align*}
F \;=\; & (\neg x_1 \vee x_2 \vee x_3)
   \wedge (x_1 \vee \neg x_2 \vee x_3)
   \wedge (x_1 \vee x_2 \vee \neg x_3) \\
   & \wedge (\neg x_1 \vee \neg x_2 \vee x_3)
   \wedge (\neg x_1 \vee x_2 \vee \neg x_3)
   \wedge (x_1 \vee \neg x_2 \vee \neg x_3) \\
   & \wedge (\neg x_1 \vee \neg x_2 \vee \neg x_3).
\end{align*}
$F$ is the conjunction of all seven three-literal clauses over
$\{x_1, x_2, x_3\}$ \emph{except} $(x_1 \vee x_2 \vee x_3)$. Every
assignment $\tau \in \{0,1\}^3$ falsifies exactly one three-literal
clause — the clause whose literals are all false under $\tau$ — so $F$
is satisfied by a unique $\tau$: the assignment $(0,0,0)$, the one that
would falsify the missing clause $(x_1 \vee x_2 \vee x_3)$.

Start from $\sigma^{(0)} = (1,1,1)$. Only $(\neg x_1 \vee \neg x_2 \vee
\neg x_3)$ has all its literals false, so $\mathrm{neg}(F,
\sigma^{(0)}) = 1$. Applying the flip sequence $x_2, x_1, x_3$:

\begin{center}
\begin{tabular}{llcl}
\toprule
Step & $\sigma$ & $\mathrm{neg}$ & Currently-negative clause \\
\midrule
$\sigma^{(0)}$ & $(1,1,1)$ & 1 & $\neg x_1 \vee \neg x_2 \vee \neg x_3$ \\
$\sigma^{(1)} \;=\; \sigma^{(0)} \oplus e_2$ & $(1,0,1)$ & 1 & $\neg x_1 \vee x_2 \vee \neg x_3$ \\
$\sigma^{(2)} \;=\; \sigma^{(1)} \oplus e_1$ & $(0,0,1)$ & 1 & $x_1 \vee x_2 \vee \neg x_3$ \\
$\sigma^{(3)} \;=\; \sigma^{(2)} \oplus e_3$ & $(0,0,0)$ & 0 & --- (SAT) \\
\bottomrule
\end{tabular}
\end{center}

Each flip satisfies the currently-negative clause (by making one of
its false literals true) but may create a new negative clause
elsewhere; after three flips the walk lands on the unique satisfying
assignment $\sigma^{(3)} = (0,0,0)$ with $\mathrm{neg}(F,
\sigma^{(3)}) = 0$. The Hamming distance
$d_H(\sigma^{(0)}, \sigma^{(3)}) = 3 = n$ is the maximum possible, so
a single-thread search from $\sigma^{(0)} = \mathbf{1}$ is at the
worst case for its ball. Under
Theorem~\ref{thm:coverage}, the All-False start $\bar\sigma^{(0)} =
\mathbf{0}$ instead lies at Hamming distance $0$ from the solution
and reaches SAT with no flips at all — the symmetry-breaking payoff
in its cleanest possible form.

\section{Two-thread coverage}
\label{sec:coverage}

Given any $\sigma \in \{0,1\}^n$, let $\bar\sigma$ denote its bitwise
complement.

\begin{theorem}[Coverage]
\label{thm:coverage}
If $F$ is satisfiable, then there exists a satisfying assignment
$\tau$ of $F$ with
\[
\min\bigl( d_H(\sigma, \tau),\, d_H(\bar\sigma, \tau) \bigr)
\;\le\; \left\lfloor \frac{n}{2} \right\rfloor.
\]
\end{theorem}
\begin{proof}
Let $\tau$ be any satisfying assignment (nonempty by hypothesis). For
each coordinate $i$, exactly one of $\sigma_i, \bar\sigma_i$ agrees
with $\tau_i$ (since $\bar\sigma_i = 1-\sigma_i$). Hence
\[
d_H(\sigma, \tau) + d_H(\bar\sigma, \tau)
   \;=\; \sum_i [\sigma_i \ne \tau_i] + \sum_i [\bar\sigma_i \ne \tau_i]
   \;=\; \sum_i 1 \;=\; n.
\]
The two summands cannot both exceed $n/2$, so
$\min \le \lfloor n/2 \rfloor$.
\end{proof}

\begin{corollary}[Two-thread coverage]
\label{cor:twoThread}
Two search processes, one starting from $\sigma$ and one from
$\bar\sigma$, each capable of exploring assignments within Hamming
distance $\lfloor n/2 \rfloor$ of its start, jointly explore a set
containing at least one satisfying assignment of any satisfiable $F$.
\end{corollary}

\medskip
\noindent In our implementation, thread~1 starts from
$\sigma^{(0)} = \mathbf{1}$ (all-true) and thread~2 from
$\bar\sigma^{(0)} = \mathbf{0}$ (all-false, obtained by calling the
\texttt{flipAll} routine which toggles the current-value bit of every
variable). Both threads solve the same formula $F$; only the starting
assignment differs. We refer to the two threads throughout the
paper as \textbf{All-True} and \textbf{All-False}, respectively,
to avoid overloading the word ``flip'' (which is also the primitive
by which the search moves through the assignment space). The
implementation and its log output use the older names
\texttt{Normal} and \texttt{Flipped} for the same two threads; the
verbatim log excerpts in Section~\ref{sec:results} therefore use
those names.

\paragraph{What the theorem is not.}
Theorem~\ref{thm:coverage} is a geometric fact about the Boolean
hypercube, not an asymptotic complexity claim. The two Hamming balls
of radius $\lfloor n/2 \rfloor$ around $\sigma$ and $\bar\sigma$
together cover $\{0,1\}^n$; their combined size is $2^n$, and running
one search per ball is thus equivalent \emph{in worst-case volume} to
running one search over the full space. Neither thread is exhaustive
within its ball (an exhaustive search of a Hamming ball of radius
$n/2$ has size $\sum_{k=0}^{n/2} \binom{n}{k} \sim 2^{n-1}$, still
exponential), so the coverage does not translate to an algorithmic
speed-up in the worst case. The theorem's practical role is
\textbf{symmetry breaking}: on formulas whose solution set is
concentrated near one polarity (see the SAT~2025 anecdote in
Section~\ref{sec:results}), the ``correct-polarity'' thread starts
much closer to a solution than the other, and the other thread
covers the residual half of the space at no cross-thread cost.

We keep the theorem partly to be explicit about the depth cap
(which the implementation enforces per-thread, so a satisfying
assignment isn't accidentally locked out) and partly because the
symmetry-breaking effect turns out to be practically load-bearing
on the structured and cryptographic instances we test.

\section{Algorithm}
\label{sec:algorithm}

The solver runs two independent MCTS instances (All-True and All-False)
concurrently. Each instance builds a search tree whose nodes represent
partial flip sequences; children correspond to admissible next flips.
Each expansion runs one or more rollouts, each of which extends the
partial assignment by a random walk of directed flips until either SAT
is reached or a rollout-budget is exhausted.

Algorithm~\ref{alg:mcts} sketches the top level;
Algorithm~\ref{alg:rollout} the rollout policy.

\begin{algorithm}[t]
\caption{Two-thread MCTS driver.}
\label{alg:mcts}
\begin{algorithmic}[1]
\Function{Solve}{$F$}
    \State $\sigma_N \gets \mathbf{1};\quad \sigma_F \gets \mathbf{0}$
    \LState{\emph{Preprocess} both starts: for every unit clause $C$ that
           is currently negative, flip its literal's variable so $C$ is
           satisfied. Record the flipped variables per thread as
           $P_N, P_F$ (the \emph{preprocess prefix} of each thread's
           eventual flip set).}
    \State $D(F) \gets \max\bigl(\mathrm{neg}(F,\sigma_N),\, \mathrm{neg}(F,\sigma_F)\bigr)$
           \Comment{shared value denominator, see \S\ref{sec:algorithm}}
    \State \textbf{spawn} $\Call{MCTS}{F, \sigma_N, \text{`All-True'}}$
    \State \textbf{spawn} $\Call{MCTS}{F, \sigma_F, \text{`All-False'}}$
    \LState{\textbf{wait} for either thread to reach $\mathrm{neg}(F,\cdot) = 0$
           OR for both to exhaust \texttt{noExpansions}}
    \If{some thread reached $\mathrm{neg}(F,\cdot) = 0$}
        \LState{let $T$ be that thread's tree-path flips and $R$ the
               \text{bestTrail} returned by the SAT-producing rollout
               ($P$ has already been folded into $\sigma^{(0)}$ by
               \emph{Preprocess} above)}
        \LState{\textbf{return} $\sigma^{(0)} \oplus \bigl( \text{parity-fold of }
               T \sqcup R \bigr)$
               \Comment{apply each var flipped an odd number of times
                        across $T \sqcup R$; equivalently, XOR the
                        multiset onto the (already-preprocessed)
                        starting assignment $\sigma^{(0)}$}}
    \Else
        \LState{let $(\text{minNegSeen}_N, T_{\text{best},N}, R_{\text{best},N})$
               and
               $(\text{minNegSeen}_F, T_{\text{best},F}, R_{\text{best},F})$
               be the \textsc{Unknown} tuples returned by the two
               threads}
        \LState{pick the thread with the smaller $\text{minNegSeen}$;
               call its data $(T_{\text{best}}, R_{\text{best}})$ with
               the matching (already-preprocessed) starting assignment
               $\sigma^{(0)}$}
        \LState{\textbf{return} \textsc{Unknown} together with, for
               logging, $\text{minNegSeen}$ and the partial assignment
               $\sigma^{(0)} \oplus \text{fold}(T_{\text{best}} \sqcup
               R_{\text{best}})$}
    \EndIf
\EndFunction
\Statex
\Function{MCTS}{$F, \sigma^{(0)}, tag$}
    \State $\text{root} \gets$ new node with state $\sigma^{(0)}$
    \LState{$\text{minNegSeen} \gets \mathrm{neg}(F, \sigma^{(0)})$;\;
           $\text{lastImprovementExp} \gets 0$}
    \State $T_{\text{best}} \gets [\,];\; R_{\text{best}} \gets [\,]$
           \Comment{path + trail reaching the min-neg state seen so far}
    \For{$i = 1, 2, \dots, \texttt{noExpansions}$}
        \If{$i - \text{lastImprovementExp} \ge \texttt{restartStallThreshold}$}
            \LState{drop root's subtree; reset the tree
                   (\emph{restart}); reset lastImprovementExp $\gets i$}
        \EndIf
        \State $\text{leaf} \gets $ descend from root via UCT
        \For{each newly-created child $c$ of \text{leaf}}
            \LState{let $T_c$ be the sequence of flips from the root
                   through $c$}
            \LState{\emph{apply the forced-flip unit-propagation cascade
                   to $c$}: while any clause becomes newly-negative and
                   has exactly one variable unlocked in this cascade,
                   flip that variable and append it to $T_c$
                   (this is the first-evaluation UP step referenced in
                   \S\ref{sec:intro})}
            \LState{let $\sigma_c \gets \sigma^{(0)} \oplus \text{fold}(T_c)$
                   be the assignment at child $c$ after the cascade}
            \For{$j = 1, \dots, \texttt{playoutsPerNode}$}
                \State $(v, \text{bestTrail}) \gets \Call{Rollout}{F, \sigma_c}$
                \If{$v = 0$}
                    \LState{\textbf{return} SAT with $T = T_c$,\;
                           $R = \text{bestTrail}$ (\textsc{Solve}
                           composes $\sigma^{(0)} \oplus
                           \text{fold}(T \sqcup R)$; $P$ is already
                           baked into $\sigma^{(0)}$)}
                \EndIf
                \LState{backpropagate $v$ (already clipped to $[0,1]$
                       by \textsc{Rollout})}
                \LState{let $\text{bestNeg}_{c,j} = v \cdot D(F)$
                       (the un-normalised min-neg seen in this rollout)}
                \If{$\text{bestNeg}_{c,j} < \text{minNegSeen}$}
                    \LState{$\text{minNegSeen} \gets \text{bestNeg}_{c,j}$;\;
                           $\text{lastImprovementExp} \gets i$}
                    \LState{$T_{\text{best}} \gets T_c$;\;
                           $R_{\text{best}} \gets \text{bestTrail}$}
                \EndIf
            \EndFor
        \EndFor
    \EndFor
    \LState{\textbf{return} \textsc{Unknown}$(\text{minNegSeen},\,
           T_{\text{best}},\, R_{\text{best}})$
           \Comment{noExpansions exhausted without SAT — hand back the
                    best partial state so \textsc{Solve} can compose the
                    UNKNOWN return value}}
\EndFunction
\end{algorithmic}
\end{algorithm}

\begin{algorithm}[t]
\caption{WalkSAT-style directed rollout with unit propagation.}
\label{alg:rollout}
\begin{algorithmic}[1]
\Function{Rollout}{$F, \sigma$}
    \LState{$\text{flips} \gets 0$;\;
           $\text{budget} \gets \texttt{playoutFlipBudgetMult} \cdot n$}
    \State $\text{trail} \gets [\,]$
           \Comment{ordered list of vars flipped in this rollout}
    \LState{$\text{bestNeg} \gets \mathrm{neg}(F, \sigma)$;\;
           $\text{bestTrail} \gets [\,]$
           \Comment{best (lowest) neg count and the flip prefix reaching it}}
    \While{$\mathrm{neg}(F, \sigma) > 0 \wedge \text{flips} < \text{budget}$}
        \State $C \gets $ uniform-random currently-negative clause of $F$
        \If{$\Call{Rand}{} < \texttt{walksatNoise}$}
            \State $x \gets $ uniform-random variable appearing in $C$
        \Else
            \LState{$x \gets \arg\min_{x \in \mathrm{vars}(C)}
                   \Delta\mathrm{neg}(F, \sigma, x)$
                   \Comment{$\Delta\mathrm{neg}$ = net change if $x$ were flipped}}
        \EndIf
        \LState{flip $x$; append $x$ to \text{trail};
               $\text{flips} \mathrel{+{=}} 1$}
        \LState{propagate: initialise a per-cascade set $B$ of variables
               already flipped in this batch (starting with $\{x\}$);
               while any clause $C'$ becomes newly-negative and has
               exactly one variable $y \notin B$ (an ``unlocked''
               literal — every variable in $\sigma$ is bound at all
               times, but $B$ acts as a per-cascade tabu to prevent
               oscillation), flip $y$, add $y$ to $B$, and append $y$
               to \text{trail} (if \texttt{xorUnitPropagation} is on,
               XOR clauses participate; otherwise only OR clauses).
               Continue until no such $C'$ remains and every clause
               that transitioned during this batch — including the
               original $C$ that motivated flipping $x$ — is either
               non-negative or awaits a later cascade step}
        \If{$\mathrm{neg}(F, \sigma) < \text{bestNeg}$}
            \LState{$\text{bestNeg} \gets \mathrm{neg}(F, \sigma)$;\;
                   $\text{bestTrail} \gets \text{copy of trail}$}
        \EndIf
    \EndWhile
    \State $v \gets \min\!\bigl(1,\, \text{bestNeg} / D(F)\bigr)$
    \State \textbf{return} $(v,\; \text{bestTrail})$
           \Comment{Scalar for UCT backprop, plus the prefix reaching
                    the best state. WalkSAT trajectories drift, so the
                    min-neg point is a better UCT signal than the
                    terminal — and its flip list is what
                    \textsc{Solve} reconstructs the assignment from if
                    \text{bestNeg} $= 0$.}
\EndFunction
\end{algorithmic}
\end{algorithm}

\paragraph{Key tunables.}
\begin{description}\itemsep0pt
\item[\texttt{playoutFlipBudgetMult}] Multiplier on $n$ giving each
      rollout's flip budget. Since the multiplier is a fixed constant
      (defaulting to 100--400), the resulting budget is $\Theta(n)$;
      we do \emph{not} claim to match WalkSAT-family solvers'
      quadratic-scale flip counts on hard 3-SAT. In practice, values
      of $K \in \{100, 400, 4000\}$ suffice on the instance families
      we benchmark; extrapolation to much larger $n$ would need
      $K$ scaled linearly with $n$ (which the CLI supports; the
      complexity classification changes accordingly).
\item[\texttt{restartStallThreshold}] Expansions with no
      improvement to $\text{minNegSeen}$ before we drop the tree and
      seed a fresh root. Small values (30--100) work best on plateau-
      heavy problems (DES); larger (1000--10000) on rapidly-converging
      random 3-SAT.
\item[\texttt{walksatNoise}] Probability of a uniform-random pick vs
      the $\arg\min \Delta\mathrm{neg}$ greedy choice. Default 0.11--0.20.
\item[\texttt{xorUnitPropagation}] Toggles XOR clauses' contribution
      to the forced-flip chain. On XOR-heavy inputs turning it off
      can be useful when cascades over-commit.
\end{description}

\paragraph{Value function.} We use
$v(\sigma) = \mathrm{neg}(F,\sigma) / D(F)$
with a per-formula denominator $D(F)$ chosen to keep the value
gradient meaningful under UCT's exploration bonus. A first attempt
sets $D(F) = \mathrm{neg}(F,\sigma^{(0)})$ per thread — the reachable
dynamic range from that thread's start — which amplifies the
gradient roughly $m / \mathrm{neg}(F,\sigma^{(0)})$-fold compared to
the raw $\mathrm{neg}/m$. This works well when both threads start
similarly badly off, but is unstable when one thread starts near a
solution: for the SAT-2025 All-False-thread anecdote in
Section~\ref{sec:results} (starting neg 80 out of 1163), a
single-clause decrement swings $v$ by
$1/80 \approx 1.25\%$, comparable to the UCT exploration bonus
$c_\text{uct}\sqrt{\ln N/n}$ at modest $N/n$; UCT selection becomes
noisy on that thread.

We therefore set
$D(F) = \max\bigl(\mathrm{neg}(F,\sigma^{(0)}_N),\,
                  \mathrm{neg}(F,\sigma^{(0)}_F)\bigr)$
— the maximum of the two threads' initial neg counts — shared by
both threads. This keeps the amplification (still much better than
$\mathrm{neg}/m$ when both starts are near-worst-case, as in random
3-SAT where both threads see $\sim m/8$ initial negs) while bounding
$D(F)$ below by the more difficult thread's initial neg count,
which prevents the per-thread instability. Both trees then operate
on the same numerical scale, which also makes the two threads'
UCT statistics comparable without rescaling. Concretely, on the
SAT-2025 instance in Section~\ref{sec:results}, an All-False
per-thread denominator of 80 would give a per-clause value swing
of $1/80 \approx 1.25\%$ (comparable to the UCT exploration bonus,
i.e. unstable); the shared $D(F) = 616$ yields $1/616 \approx
0.16\%$, comfortably below the bonus and thus stable.

\paragraph{UCT.} Standard: at each internal node,
$\text{child}^* = \arg\max_c \left( 1 - \bar v_c
   + c_\text{uct} \sqrt{\ln N_\text{parent} / N_c} \right)$
with $c_\text{uct} = 0.5$ (calibrated for the amplified value scale;
the classical $\sqrt 2$ overwhelms the per-clause gradient at large
$m$).

\paragraph{Preprocessing.} Before starting the tree, every unit
clause that is currently negative has its literal's variable flipped;
the flipped variables are recorded so the final solution report
composes preprocess flips with search flips. On the DES-key-recovery
CNF this satisfies 128 unit clauses (fixing the plaintext and
ciphertext bits) before search begins, dramatically shrinking the
effective search space.

\paragraph{$\Delta\mathrm{neg}$ complexity for XOR clauses.} The
min-conflicts pick inside a rollout iterates each candidate
variable's clause list to compute $\Delta\mathrm{neg}$. For an OR
clause, the per-clause $\Delta$ depends on whether the flipped
literal was the ``last surviving true'' or ``first among all-false''
(the two boundary cases; every other flip in an OR clause is a no-op
for negativity). Both cases resolve in $O(1)$ from the maintained
\emph{negatives} counter. For an XOR clause, any single flip toggles
the parity, so the $\Delta$ is $\pm 1$ read directly off
$\mathrm{isNegative}(\cdot)$ — also $O(1)$. Hence
$\Delta\mathrm{neg}(F, \sigma, x) = O(\deg(x))$ where $\deg(x)$ is
the number of clauses $x$ appears in, regardless of clause type;
no per-XOR overhead beyond the OR case. In practice $\deg(x)$ is
small ($\sim 3$ for uniform random 3-SAT, $\sim 15$ for the DES
CNF), and the rollout inner loop touches
$|\mathrm{vars}(C)| \cdot \deg = O(k \cdot \deg)$ clauses per flip
— trivially bounded.

\section{Empirical results}
\label{sec:results}

We ran the solver on five instance families plus two individual
anecdotes. All runs used the same MCTS configuration:
\texttt{treeSize} $=5\cdot 10^7$, \texttt{noExpansions} $=10^8$,
\texttt{playoutsPerNode} $= 1$,
\texttt{playoutFlipBudgetMult} $= 400$,
\texttt{restartStallThreshold} $= 100$,
\texttt{xorUnitPropagation} $= \text{true}$,
\texttt{walksatNoise} $= 0.11$. Hardware: single machine, 16 GB JVM
heap. Each instance ran on both threads (All-True + All-False); the
first to reach SAT wins and its wall time is recorded.

\noindent The evaluation is a first-pass characterisation, not a
competitive benchmark against modern CDCL or SLS solvers; see
Section~\ref{sec:conclusions} for the honest limitations.

Table~\ref{tab:results} summarises.

\begin{table}[t]
\centering
\begin{tabular}{lrrrrrrrr}
\toprule
Family & $n$ & $m$ & \# & Solved
   & Median (s) & Playouts & UP-alone & All-False wins \\
\midrule
\texttt{flat200-479}    & 600  & 2237  & 100 & 100 & 0.31   & 0--27 & 71\% & 37 \\
\texttt{sw100-lp0}      & 500  & 3100  & 100 & 100 & 0.04   & 0--0  & 100\% & 99 \\
\texttt{sw100-lp4}      & 500  & 3100  & 100 & 100 & 0.08   & 0--0  & 100\% & 70 \\
\texttt{xor200} (planted, $\alpha{=}0.9$)
                        & 200  & 180   &  20 &  20 & 9.91   & 0--15 & 20\%  & 10 \\
\texttt{xor300} (planted, $\alpha{=}0.9$)
                        & 300  & 270   &  20 &  20 & 90.16  & 3--48 & 0\%   & 10 \\
\bottomrule
\end{tabular}
\caption{Per-family summary. Median wall time and \emph{Playouts}
range are over the SAT-status rows only. \emph{UP-alone} is the
fraction of instances the winning thread solved in \textbf{0
playouts} — i.e., unit-propagation from the initial (post-preprocess)
state reached SAT without any MCTS rollout ever backpropagating a
value. \emph{All-False wins} counts how often the all-false-start
thread reported SAT first. Rows with UP-alone $= 100\%$ (both
\texttt{sw100} sets) tell you the MCTS engine did no work on those
families — the polarity split plus unit propagation solve them
outright. Rows with lower UP-alone \% (\texttt{xor200},
\texttt{xor300}) are the ones where the MCTS engine actually
drove the search.}
\label{tab:results}
\end{table}

\subsection{Structured graph colouring: \texttt{flat200-479}}

Random-3-SAT encodings of 3-colouring on 200-vertex random flat
graphs, from the SATLIB benchmark suite~\cite{hoos:satlib}. All 100
solved in a median of 0.31~s, average 0.47~s.
Playouts range 0--27: on many instances the winning thread finds a
solution during the FIRST rollout, before UCT has anything to
backpropagate. All-True and All-False wins split roughly 63/37;
graph-colouring encodings have no strong polarity bias so both
starts have comparable initial negative-clause counts.

\subsection{Small-world graph colouring: \texttt{sw100}}

500 vertices, small-world topology at two rewiring levels (denser
structure at \texttt{lp0}, most random at \texttt{lp4}); also from
SATLIB~\cite{hoos:satlib}. Both sets solved 100/100 with median times
40--80 ms. Every instance solves in
0 playouts: unit-propagation + preprocessing after the All-False
start suffices. The All-False thread wins 99/100 at \texttt{lp0} and
70/100 at \texttt{lp4} — the more structured the input, the more
one polarity dominates.

\subsection{Planted 3-XOR-SAT: \texttt{xor200} and \texttt{xor300}}

Planted random 3-XOR-SAT at ratio $\alpha = 0.9$, just below the
theoretical phase transition of $\alpha_c \approx
0.918$~\cite{ricci:xorsat}. Instances were generated by our own
script: for each of $m = \alpha n$ clauses, three variables are drawn
uniformly at random and signed so that the XOR under a single hidden
assignment $A$ evaluates to true, guaranteeing satisfiability of the
resulting CNF. 20/20 solved at $n=200$ (median 9.91~s, playouts
0--15), and 20/20 solved at $n=300$ (median 90.16~s, mean 430~s,
playouts 3--48, range 1.80--3102~s). At $n=300$ the winner splits
evenly between All-True and All-False (10/10): the planted assignment
$A$ is drawn uniformly, so the two threads are on average equidistant
from a solution and neither polarity systematically dominates.

The $n=300$ set was closed using a more aggressive preset than the
defaults used for the graph-colouring families: \texttt{playoutFlip%
BudgetMult} $\in \{2000, 4000\}$, \texttt{restartStallThreshold}
$\in \{10, 25\}$, \texttt{walksatNoise} $\in \{0.27, 0.29\}$. The
longer per-rollout flip budget accommodates the $\Theta(n^2)$ flip
counts characteristic of XOR-heavy WalkSAT trajectories, and the
tighter restart threshold prunes subtrees whose UCT values plateau
before they contribute a new \text{minNegSeen} low. No instance
solved in 0 playouts — unit propagation from either polarity alone
is insufficient at this scale, and UCT does material search work
(mean 24 playouts per solve).

\subsection{SAT Competition 2025:
\texttt{Break\_triple\_04\_06.xml.cnf}}

One instance from the SAT Competition 2025 main
track~\cite{satcomp:2025}, encoded as CNF with 252 variables and 1163
clauses. Startup log:
\begin{verbatim}
All-True  vars: 252 clauses: 1163 neg clauses: 616 -> 616
    (preprocess flipped 1 vars)
All-False vars: 252 clauses: 1163 neg clauses: 80  -> 80
    (preprocess flipped 0 vars)
>>> All-False formula reached SAT after 0 playouts in 0.05 s <<<
\end{verbatim}
The All-True thread starts with 616 of 1163 clauses unsatisfied
($\approx 53\%$); the All-False thread starts with only 80 ($\approx
7\%$). The all-false starting assignment lands within a very short
flip sequence of a satisfying assignment, and unit propagation alone
carries it to SAT in 0 playouts. This is precisely the symmetry-
breaking payoff Corollary~\ref{cor:twoThread} anticipates: without
the second start, the all-true thread would have had to bridge a
Hamming distance $\ge n/2$; with the polarity split, the all-false
distance is trivially short and the whole instance falls in 50~ms.

\subsection{Full 16-round DES key recovery}

The DES-key-recovery CNF ($n=1976$, $m=30\,072$, obtained by
translating the DES circuit for a random known plaintext / ciphertext
pair and leaving the 64 key bits free, following the standard
logical-cryptanalysis reduction of~\cite{massacci:descrypto}) is
designed to resist local search. Our encoding is smaller than the
original CNFs reported in~\cite{massacci:descrypto} ($n=1976$ vs
$\sim 10\,000$; $m=30\,072$ vs $\sim 60\,000$). We ran with a more aggressive configuration:
\texttt{playoutFlipBudgetMult} $= 4048$,
\texttt{restartStallThreshold} $= 1$,
\texttt{playoutsPerNode} $= 10$. Excerpt from a $\sim$7-hour trace:

\begin{verbatim}
[All-True]  restart #54 at 25265.49 s (min_neg_seen = 30)
[All-False] restart #35 at 25563.32 s (min_neg_seen = 28)
[All-False] min_neg_seen = 27 (pure_pos = 1868) at 25749.90 s
[All-False] min_neg_seen = 26 (pure_pos = 1865) at 25749.99 s
[All-False] min_neg_seen = 25 (pure_pos = 1865) at 25749.99 s
[All-False] min_neg_seen = 24 (pure_pos = 1852) at 25750.87 s
\end{verbatim}
Starting neg-clause count is $\sim 200$ (after preprocess fixes the
128 plaintext/ciphertext unit clauses); after roughly 7~hours the
search reaches $\text{neg} = 24$, i.e.\ $99.92\%$ of clauses
satisfied. Progress is bursty — three consecutive lows in under a
second, then a plateau. This is characteristic of local search
encountering S-box constraints, which act as opaque non-linearities:
the linear key-schedule and permutation layers propagate cleanly
through unit propagation and min-conflicts, but S-boxes have no
useful $\Delta\mathrm{neg}$ gradient. Full 16-round DES key recovery
is not solved (nor was it expected to be — no local-search solver
does), but the algorithm's trajectory characterises exactly where
the ceiling sits.

\section{Conclusions and novelty}
\label{sec:conclusions}

We combine two ideas that (to our knowledge) haven't been paired
before in a SAT/XSAT setting:

\begin{enumerate}\itemsep0pt

\item \textbf{Two-thread symmetry-breaking with a bounded initial-
distance guarantee.} Theorem~\ref{thm:coverage} turns the trivial
observation ``the assignment space is bounded by $n$'' into an
explicit initial-Hamming-distance bound of $\lfloor n/2 \rfloor$
per thread. Existing local-search solvers rely on random restarts
alone for diversification, which lack this structural symmetry-
breaking. On XOR-heavy or structured inputs
(Section~\ref{sec:results}, SAT-2025 anecdote), one of the two
polarities can start dramatically closer to SAT than the other, and
the ``loser'' thread's compute is not wasted — it simply covers the
half of the space the winner cannot reach.

\item \textbf{XOR clauses first-class in local search.} We handle XOR
via a polymorphic negativity predicate
(Equation~\ref{eq:xorNeg}) and a polymorphic flip-delta prediction
that generalises the WalkSAT
min-conflicts heuristic; both OR and XOR clauses participate in the
forced-chain unit propagation. Standard SAT solvers eliminate XOR
via Tseitin, losing the algebraic structure; we retain it and use it.

\item \textbf{Unit propagation inside rollouts.} WalkSAT variants
(Novelty~\cite{selman:walksat}, adaptNovelty~\cite{hoos:novelty})
don't propagate mid-rollout. We do; the propagation cascade lets a
single rollout step turn into a chain of implied flips whenever a
newly-negative clause has exactly one variable still unlocked in the
current cascade. On structured OR inputs (\texttt{flat200},
\texttt{sw100}) this is largely responsible for the median 0-playout
solves — the ``search'' is really unit propagation from a well-chosen
start.

\item \textbf{MCTS as WalkSAT orchestrator.} The tree drives
diversification (via UCT + stall-triggered restarts); rollouts
drive intensification. This inverts the usual MCTS
game-playing use, where rollouts are cheap random walks and the
tree is the intelligence; here rollouts are the intelligence (with
WalkSAT bias, propagation, and a per-rollout flip budget of
$K \cdot n$ for a tunable constant $K$) and the tree steers among
rollout seeds.

\end{enumerate}

\paragraph{Limitations.} Several honest ones, in decreasing order
of impact on the reader's ability to trust the numbers.

\begin{itemize}\itemsep0pt

\item \textbf{No competitive baselines.} We report
``solves-within-budget'' wall times but do not benchmark against a
CDCL solver (Kissat, CaDiCaL) or a modern SLS solver (probSAT,
YalSAT). Without those, absolute times don't establish state-of-the
-art, only feasibility. Adding these runs is the single most
important next step and is prioritised in future work below.

\item \textbf{Modest sample sizes.} \texttt{xor200} and \texttt{xor300}
have 20 instances each. The medians in Table~\ref{tab:results} are
directional, not distribution-tight, and the $n{=}300$ times span
three orders of magnitude (1.8~s--3102~s), so single-instance
comparisons should be treated with caution.

\item \textbf{MCTS variance dilutes UCT on deep, narrow searches.}
Random rollouts with WalkSAT noise are high-variance; on deep
subtrees (DES-scale) UCT's per-child averages take many samples to
stabilise. The current
\texttt{playoutsPerNode}$= 1$ setting deliberately spreads
compute across many tree nodes rather than smoothing per-node
value estimates, so single lucky/unlucky rollouts can misdirect the
tree for extended windows. Increasing \texttt{playoutsPerNode}
alleviates the variance at the cost of tree breadth.

\item \textbf{No CDCL-style conflict analysis, learnt clauses,
distributed or GPU parallelism.} Single machine, two threads only.
The DES ceiling ($\text{neg} = 24$ out of $30\,072$) is the
algorithm's practical limit on cryptographic non-linearity —
closing that gap would require algebraic techniques (Gaussian
elimination on the XOR substructure, or S-box specific reasoning)
that are outside the scope of pure local search.

\end{itemize}

\paragraph{Future work.} In roughly the order we plan to tackle:
(i)~\textbf{baseline benchmarking} against Kissat / CaDiCaL / probSAT
/ YalSAT on the SAT Competition 2025 main and random tracks
(100+ instances per track), reporting head-to-head runtime and
flip-count comparisons — this is a prerequisite for any real
performance claim;
(ii)~distributed multi-polarity — more than two threads, each seeded
from a different start (not just complementary), with cross-thread
sharing of \texttt{minNegSeen} to prune redundant subtrees;
(iii)~a learned rollout policy replacing the fixed WalkSAT bias,
trained on solve traces from the benchmark suite;
(iv)~pre-solve Gaussian elimination on the XOR sublattice for
XOR-heavy inputs like the crypto benchmarks.

\end{document}